\documentclass[a4paper,12pt]{article}
\usepackage{amssymb,amsmath,amsthm,latexsym}
\usepackage{amsfonts}
	\usepackage{amsfonts}
\usepackage{graphicx}
\usepackage[pdftex,bookmarks,colorlinks=false]{hyperref}
\usepackage{verbatim}
\usepackage[font=small,labelfont=bf]{caption}
\usepackage{subcaption}
\usepackage[hmargin=1in,vmargin=1in]{geometry}

\newtheorem{theorem}{Theorem}[section]

\newtheorem{corollary}[theorem] {Corollary}
\newtheorem{definition}[theorem]{Definition}

\newtheorem{lemma} [theorem]{Lemma}

\newtheorem{proposition}[theorem]{Proposition}
\newtheorem{remark}[theorem]{Remark}

\title{\bf A Note on the Sparing Number of Graphs}
\author{{\bf N K Sudev \footnote{Department of Mathematics, Vidya Academy of Science \& Technology, Thalakkottukara, Thrissur - 680501, email: {\em sudevnk@gmail.com}}} and {\bf K A Germina\footnote{Department of Mathematics, School of Mathematical \& Physical Sciences, Central University of Kerala, Kasaragod, email:{\em srgerminaka@gmail.com}}}}
\date{}
\begin{document}
\maketitle

\begin{abstract}
An integer additive set-indexer is defined as an injective function $f:V(G)\rightarrow 2^{\mathbb{N}_0}$ such that the induced function $g_f:E(G) \rightarrow 2^{\mathbb{N}_0}$ defined by $g_f (uv) = f(u)+ f(v)$ is also injective. An IASI $f$ is said to be a weak IASI if $|g_f(uv)|=max(|f(u)|,|f(v)|)$ for all $u,v\in V(G)$. A graph which admits a weak IASI may be called a weak IASI graph. The set-indexing number of an element of a graph $G$, a vertex or an edge, is the cardinality of its set-labels. The sparing number of a graph $G$ is the minimum number of edges with singleton set-labels, required for a graph  $G$ to admit a weak IASI.  In this paper, we study the sparing number of certain graphs and the relation of sparing number with some other parameters like matching number, chromatic number, covering number, independence number etc.
\end{abstract}

\section{Introduction}

For all  terms and definitions, not defined specifically in this paper, we refer to \cite{FH}. Unless mentioned otherwise, all graphs considered here are simple, finite and have no isolated vertices.

Let $\mathbb{N}_0$ denote the set of all non-negative integers. For all $A, B \subseteq \mathbb{N}_0$, the sum of these sets is denoted by  $A+B$ and is defined by $A + B = \{a+b: a \in A, b \in B\}$. The set $A+B$ is called the sumset of the sets $A$ and $B$. 

An {\em integer additive set-indexer} (IASI, in short) is defined in \cite{GA} as an injective function $f:V(G)\rightarrow 2^{\mathbb{N}_0}$ such that the induced function $g_f:E(G) \rightarrow 2^{\mathbb{N}_0}$ defined by $g_f (uv) = f(u)+ f(v)$ is also injective. 

The cardinality of the labeling set of an element (vertex or edge) of a graph $G$ is called the {\em set-indexing number} of that element.

\begin{lemma}\label{L-Card}
\cite{GS1} Let $A$ and $B$ be two non-empty finite subsets of $\mathbb{N}_0$. Then, $max(|A|,|B|)\\ \le |A+B|\le |A|.|B|$. Therefore, for an integer additive set-indexer $f$ of a graph $G$, we have $max(|f(u)|, |f(v)|)\le |g_f(uv)|= |f(u)+f(v)| \le |f(u)| |f(v)|$, where $u,v\in V(G)$.
\end{lemma}

\begin{definition}{\rm
\cite{GS1} An IASI $f$ is said to be a {\em weak IASI} if $|g_f(uv)|=max(|f(u)|,|f(v)|)$ for all $u,v\in V(G)$. A graph which admits a weak IASI may be called a {\em weak IASI graph}. A weak  IASI $f$ is said to be {\em weakly $k$-uniform IASI} if $|g_f(uv)|=k$, for all $u,v\in V(G)$ and for some positive integer $k$.}
\end{definition}

\begin{definition}{\rm
\cite{GS3} An element (a vertex or an edge) of graph which has the set-indexing number 1 is called a {\em mono-indexed element} of that graph.}
\end{definition}

\begin{definition}{\rm
\cite{GS3} The {\em sparing number} of a graph $G$ is defined to be the minimum number of mono-indexed edges required for $G$ to admit a weak IASI and is denoted by $\varphi(G)$.}
\end{definition}

\section{New Results on the Sparing Number of Graphs}

First, we recall the following theorems proved in \cite{GS3}. 

\begin{theorem}\label{T-SB1}
\cite{GS1} If a connected graph $G$ admits a weak IASI, then $G$ is bipartite or $G$ has at least one mono-indexed edge.
\end{theorem}

\begin{theorem}\label{T-SB2}
\cite{GS1} Let $G$ be a bipartite graph which admits a weak IASI and let $u,v$ be two non-adjacent vertices in $G$. Then, $G\cup \{uv\}$ is a weak IASI graph if and only if $G\cup \{uv\}$ is bipartite or $uv$ is a mono-indexed edge.
\end{theorem}

Due to Theorem \ref{T-SB1}and Thorem \ref{T-SB2}, we have the following theorem.

\begin{theorem}\label{T-SNGB1}
Let $G$ be a weak IASI graph and let $E'$ be a minimal subset of $E(G)$ such that $G-E'$ is bipartite. Then, $\varphi(G)=|E'|$. 
\end{theorem}
\begin{proof}
Let $G$ be a non-bipartite weak IASI graph. Let $E'=\{e'_1,e'_2, \ldots, e'_r\}$ be a minimal set of edges of $G$ such that $G_r=G-E'$ is bipartite. Then, $G_r=G-E'$ is a maximal bipartite subgraph of $G$. By Theorem \ref{T-SB1}, $G_r$ admits a weak IASI without any mono-indexed edge. Now, $G_{r-1}=G_n\cup \{e'_i\}, 1\le i \le r$ is a non-bipartite graph. Then, by Theorem \ref{T-SB2}, $e'_i$ is a mono-indexed edge in $G$, for $1\le i \le r$. Hence, in each edge in $E'$ must be mono-indexed. Therefore, $\varphi(G)=|E'|$.
\end{proof}

The size of the maximal bipartite subgraph of a graph $G$ is denoted by $b(G)$. Then, as a consequence of Theorem \ref{T-SNGB1}, we have

\begin{corollary}\label{C-T-SNGB1}
For a graph $G$, the size of its maximal bipartite subgraph is $b(G)=|E(G)|-\varphi(G)$.
\end{corollary}
\begin{proof}
Let $G$ be a non-bipartite graph with the edge set $E$. Let $E'\subset E$ such that $G-E'$ is the maximal bipartite subgraph of $G$. Then, $b(G)=|E(G-E')|=|E|-|E'|$. By Theorem \ref{T-SNGB1}, we have $|E'|=\varphi(G)$. Therefore, $b(G)=|E(G)|-\varphi(G)$.
\end{proof}

The most interesting question in this context is whether the sparing number of a graph is related to other parameters like chromatic number, matching number etc. The following results establish some relations between these parameters of certain graph classes.

Let $G$ be a $k$-chromatic graph with colors $\mathtt{c}_i, 1\le i \le k$.  Then, $G$ is a $k$-partite graph in which the $i$-th partition, denoted by $\mathcal{C}_i, 1\le i \le k$, consists of all the vertices of $G$ which has the color $\mathtt{c}_i$. The set $\mathcal{C}_i$ is called the {\em $i$-th color class}. Hence we have, 

\begin{theorem}
The sparing number of a $k$-chromatic graph $G$ is the total number of vertices in all the color classes of $G$ except the two maximal color classes.
\end{theorem}
\begin{proof}
Let $\mathtt{c}_1, \mathtt{c}_2, \mathtt{c}_3, \ldots, \mathtt{c}_k$ be the $k$ colors that are used for coloring the vertices $G$ and let $\mathcal{C}_i$ be the color class in which the vertices are assigned the color $\mathtt{c}_i$. Without loss of generality, let $\mathcal{C}_i, 1 \le i \le k$ be in the descending order of their cardinality. Label all the vertices in the color class $\mathcal{C}_1$ by distinct singleton sets of non-negative integers and label all the vertices in the color class $\mathcal{C}_2$ by non-singleton singleton sets of non-negative integers. Since $\mathcal{C}_1$ and $\mathcal{C}_2$ are maximal color classes, each vertex, say $v$, in other color classes must be adjacent to at least one vertex of $\mathcal{C}_1$ and at least one vertex of $\mathcal{C}_2$. Since $G$ is a weak IASI graph, $v$ can be labeled only by a singleton set, which is not used for labeling before. Hence, the sparing number of $G$ is the total number of vertices in all the color classes of $G$ except the two maximal color classes.
\end{proof}


Now, recall the following theorem proved in \cite{GS3}.

\begin{theorem}\label{T-SNC1}
\cite{GS3} The sparing number of an odd cycle is $1$ and that of a bipartite graph is $0$.
\end{theorem}

Due to Theorem \ref{T-SNC1}, we have,

\begin{theorem}
If $G$ is path or cycle, then $\chi(G)-\varphi(G)=2$, where $\chi(G)$ is the chromatic number of $G$.
\end{theorem}
\begin{proof}
If $G$ is path or an even cycle, then $\chi(G)=2$ and by Theorem \ref{T-SNC1}, $\varphi(G)=0$. Hence, $\chi(G)-\varphi(G)=2$. If $G$ is an odd cycle, then $\chi(G)=3$ and by Theorem \ref{T-SNC1}, $\varphi(G)=1$. Hence, $\chi(G)-\varphi(G)=2$. This completes the proof.
\end{proof}

 A {\em matching} $M$ in a given graph $G$ is a set of pairwise non-adjacent edges. A {\em maximum matching} is a matching that contains the largest possible number of edges. The matching number $\nu(G)$ of $G$ is the size of a maximum matching. The following theorem establishes the relation between the sparing number and the matching number paths and cycles.

\begin{theorem}\label{T-SNMNCP}
If $G$ is a path or a cycle on $n$ vertices, then $\varphi(G)=\lceil \frac{n}{2} \rceil-\nu(G)$, where $\nu(G)$ is the matching number of $G$. 
\end{theorem}
\begin{proof}
If $G$ is a path, then by Theorem \ref{T-SNC1}, $\varphi(G)=0$. If $G$ is even, then $\nu(G)=\frac{n}{2}$ and if $G$ is odd, then $\nu(G)=\frac{(n+1)}{2}$. Therefore, the matching number of $G$ is $\nu=\lceil \frac{n}{2} \rceil$. Therefore, $\varphi(G)=\lceil \frac{n}{2} \rceil-\nu(G)$.

If $G$ is an even cycle, then by Theorem \ref{T-SNC1}, $\varphi(G)=0$ and the matching number of $G$ is $\nu(G)=\frac{n}{2}$. Here, $\frac{n}{2}-\nu(G)=\varphi(G)$.

If $G$ is an odd cycle, then by Theorem \ref{T-SNC1}, $\varphi(G)=1$ and the matching number of $G$ is $\nu(G)=\frac{n-1}{2}$. Here, $\lceil {\frac{n}{2}} \rceil-\nu(G)=1=\varphi(G)$.

This completes the proof.
\end{proof}

\noindent Now, recall the following theorem.

\begin{theorem}\label{T-SNUG}
\cite{GS4} Let $G_1$ and $G_2$ be two weak IASI graphs. Then, $\varphi(G_1\cup G_2) = \varphi(G_1)+\varphi(G_2)-\varphi(G_1\cap G_2)$.
\end{theorem}

As a consequence of Theorem \ref{T-SNMNCP} and Theorem \ref{T-SNUG}, we have

\begin{theorem}\label{T-SNFUOC}
If $G$ can be decomposed into finite number of odd cycles, then $\varphi(G)=\displaystyle{\sum_{i=1}^m{\lceil{\frac{n_i}{2}}\rceil}} - \nu(G)$, where $m$ is the number of edge-disjoint cycles and $n_i$ is the size of the cycle $C_i$ in $G$.
\end{theorem}
\begin{proof}
Let $G$ can be decomposed into finite number of odd cycles. Let $C_1,C_2,C_3,\ldots,C_m$ be the distinct edge disjoint odd cycles in $G$ so that $\displaystyle{\bigcup_{i=1}^m C_i}=G$. Then, 
\begin{equation}
\nu(G)=\displaystyle{\sum_{i=1}^m{\nu(C_i)}} \label{myeq1}
\end{equation}
 
By Theorem \ref{T-SNMNCP}, $\varphi(C_i)=\lceil{\frac{n_i}{2}}\rceil - \nu(C_i)$. Since, all $C_i, 1\le i \le m$ are edge-disjoint, by Theorem \ref{T-SNUG}, 
	\begin{eqnarray*}
	\varphi(G) & = & \displaystyle{\sum_{i=1}^m{\varphi(C_i)}}\\
	 & = & \displaystyle{\sum_{i=1}^m{[\lceil{\frac{n_i}{2}}\rceil - \nu(C_i)]}}\\
	 & = & \displaystyle{\sum_{i=1}^m{\lceil{\frac{n_i}{2}}\rceil}} - \displaystyle{\sum_{i=1}^m{\nu(C_i)}}\\
	 & = & \displaystyle{\sum_{i=1}^m{\lceil{\frac{n_i}{2}}\rceil}} - \nu(G)
	\end{eqnarray*}
This completes the proof.
\end{proof}

\begin{remark}\label{R-SNC1}{\rm
We observe that Theorem \ref{T-SNFUOC}, does not hold for edge disjoint bipartite graphs since Equation \ref{myeq1} does not hold for edge disjoint even cycles. It can also be verified that Equation \ref{myeq1} will not hold for those graphs, which can be decomposed into edge disjoint cycles, having more than one even cycle. }
\end{remark}

Now, recall the well-known theorem on Eulerian graphs.

\begin{theorem}\label{T-EG1}
\cite{CZ} A graph $G$ is Eulerian if and only if it can be decomposed into edge disjoint cycles.
\end{theorem}

In view of Theorem \ref{T-EG1} and Remark \ref{R-SNC1}, Theorem \ref{T-SNFUOC} can be rewritten as follows.

\begin{theorem}
If $G$ is an Eulerian graph that has at most one even cycle then $\varphi(G)=\displaystyle{\sum_{i=1}^m{\lceil{\frac{n_i}{2}}\rceil}} - \nu(G)$, where $m$ is the number of edge-disjoint cycles and $n_i$ is the size of the cycle $C_i$ in $G$.
\end{theorem}

A {\em vertex cover} of a graph $G$ is a subset $S$ of $V(G)$ such that each edge of $G$ has at least one end vertex in $S$. The number of vertices in a minimum vertex cover of a graph $G$ is known as the {\em vertex covering number} or simply the {\em covering number} of $G$ and is denoted by $\beta(G)$.

The relations of weak IASI and sparing number with its minimal vertex cover and covering number are discussed in the following theorem.

\begin{theorem}\label{T-SNVC}
The minimum number of mono-indexed vertices in a weak IASI graph $G$ is equal to the covering number of $G$. Moreover, the sparing number of $G$ is the number edges of $G$ which have both of their end vertices in the minimal vertex cover of $G$.
\end{theorem}
\begin{proof}
Let $S$ be the minimal vertex cover of $G$. Then, $S$ and $V-S$ are two partitions of $V(G)$. Since every edge of $G$ has at least one end vertex in $S$, no two vertices in $V-S$ are adjacent in $G$ and some edges in $G$ may have both of their end vertices in $S$. Since $G$ is a weak IASI graph, the vertices in $S$ can not be labeled by non-singleton vertices. Hence, the set-labels of the vertices in $G$ are singleton sets of non-negative integers. That is, the number of mono-indexed vertices in $G$ is equal to its covering number.

Since no two vertices in $V-S$ are adjacent each other, we can label all the vertices in $V-S$ by non-singleton sets of non-negative integers. Hence, an edge having both of its end vertices in $S$ is mono-indexed. Therefore, the sparing number of $G$ is the number of edges of $G$ which have both of their end vertices in $S$. This completes the proof.
\end{proof}

An {\em independent set} or {\em stable set} of a graph $G$ is a subset $S'$ of $V(G)$ in a graph, no two vertices in $S'$ are adjacent. It is to be noted that a subset $S'$ of $V(G)$ is independent if and only if its complement $V(G)-S'$ is a vertex cover.  The number of vertices in a maximal independence set is called the independence number of $G$ and is denoted by $\alpha(G)$.

The following theorem establishes the relation between the independence number and covering number of a given graph $G$.

\begin{theorem}\label{T-RINCN}
\cite{FH} For any connected non-trivial graph $G$, $\alpha(G)+\beta(G)=|V(G)|$.
\end{theorem}

Note that if $S$ is a vertex cover of a graph $G$, then the set $V-S$ is an independent set of $G$. Hence, if $S$ is a minimal vertex cover of $G$, then $V-S$ is a maximal independent set of $G$. Hence we have,

\begin{theorem}
Let $G$ be a weak IASI graph on $n$ vertices. Then, the number of mono-indexed vertices in $G$ is $n-\alpha(G)$, where $\alpha(G)$ is the independence number of the graph $G$.
\end{theorem}
\begin{proof}
By Theorem \ref{T-SNVC}, the number of mono-indexed vertices of a graph $G$ is equal to $\beta(G)$, the covering number of $G$. By Theorem \ref{T-RINCN}, we have $\beta(G)=n-\alpha(G)$. Therefore, the number of mono-indexed vertices = $n-\alpha(G)$.
\end{proof}

If $S$ is a minimal vertex cover of $G$ and if all vertices in $S$ are labeled by distinct singleton sets of non-negative integers, $V-S$ is a maximal independent set in which all elements are non-adjacent to each other and the set-labels of all of them are distinct non-singleton sets of non-negative integers. Hence, we have the following proposition.

\begin{proposition}
If $G$ is a weak IASI graph, then the maximum number of vertices that are not mono-indexed in $G$ is equal to the independence number of $G$.
\end{proposition}

\newpage

\section{The Sparing Number of Certain Named Graphs}

In view of Theorem \ref{T-SNGB1}, we discuss the sparing number of some standard non-bipartite graphs.

\begin{proposition}
The sparing number of Petersen graph is $3$.
\end{proposition}
\begin{proof}
Name the vertices of the Petersen Graph $G$ as shown in Figure \ref{G-PGWL}. In $G$, both the external cycle $C_1:u_1u_2u_3u_4u_5u_1$ and the internal cycle $C_2:v_1v_3v_5v_2v_4v_1$ are odd and hence $G$ is not bipartite. Hence, remove one edge, say $u_3u_4$, from the cycle $C_1$ and remove one edge, say $v_2v_5$, from the cycle $C_2$. 

In the resultant graph $G'=G-\{u_3u_4,v_2v_5\}$, the cycles $C_3:u_1v_1v_4u_4u_5u_1$ and $C_4:u_1v_1v_3u_3u_2u_1$ are odd. Hence, remove the common edge $u_1v_1$ from $C_3$ and $C_4$. Let $H=G-\{u_3u_4,v_2v_5,u_1v_1\}$, which is shown in Figure \ref{G-PG-MBSG}. All the cycles in $H$ are of even length and hence $H$ is the maximal bipartite subgraph of $G$. Let $E'=\{u_3u_4,v_2v_5,u_1v_1\}$. Then, by Theorem \ref{T-SNGB1}, $\varphi(G)=|E'|=3$.
\end{proof}

\begin{figure}[h!]
\begin{center}
\begin{subfigure}[b]{0.45\textwidth}
\includegraphics[scale=0.45]{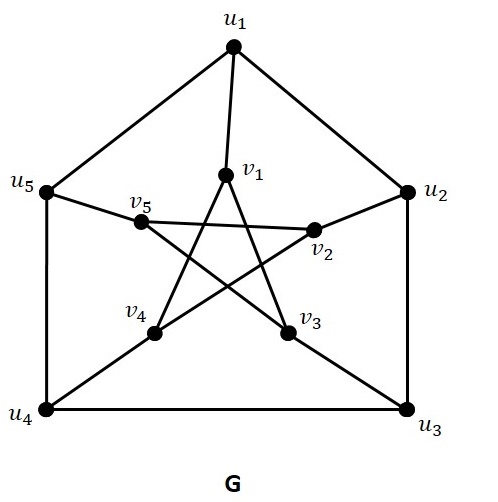}
\caption{\small \sl The Peterson graph.}\label{G-PGWL}
\end{subfigure}
\quad
\begin{subfigure}[b]{0.45\textwidth}
\includegraphics[scale=0.45]{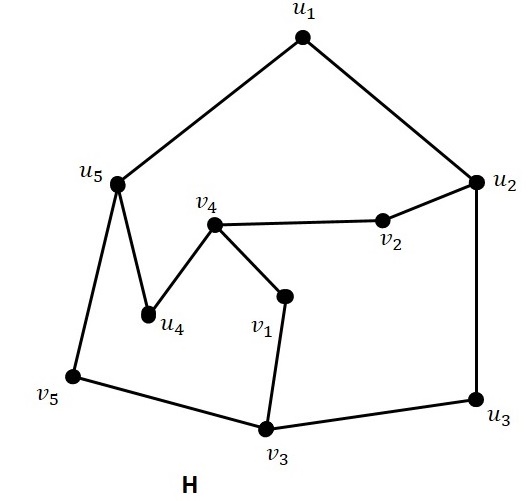}
\caption{\small \sl Maximal bipartite subgraph of the Peterson graph.}\label{G-PG-MBSG}
\end{subfigure}
\caption{}
\end{center}  
\end{figure}

\begin{remark}{\rm
For the Peterson graph $G$, $|E(G)|=15$. Hence, by Corollary \ref{C-T-SNGB1}, the number of edges in a maximal bipartite subgraph of $G$ is $b(G)= |E(G)|-\varphi(G)=15-3=12$.}
\end{remark}

One other well-known non-bipartite graph is the {\em Frucht graph}, which is is a $3$-regular graph with $12$ vertices, $18$ edges and every vertex of which can be distinguished topologically from every other vertex. Hence we have,

\begin{proposition}
The sparing number of Frucht graph is $3$.
\end{proposition}
\begin{proof}
Name the vertices of the Frucht Graph $G$ as shown in Figure \ref{G-FGWL}. In $G$, the cycle $v_5v_6v_{12}v_5$ is of odd length and has a common edge $v_5v_{12}$ with the even cycle $v_5v_{12}v_{11}v_9v_{10}v_4v_5$ and has a common edge $v_6v_12$ with the even cycle $v_6v_{12}v_{11}v_7v_6$. Hence, remove the edge $v_5v_6$ from $G$. 
The cycle $v_3v_4v_{10}v_3$ in $G-\{v_5v_6\}$ has a common edge $v_4v_{10}$ with the even cycle $v_5v_{12}v_{11}v_9v_{10}v_4v_5$ and has a common edge $v_3v_{10}$ with the odd cycle $v_2v_3v_{10}v_9v_8v_2$. Hence, remove the edge $v_4v_{10}$ from $G-\{v_5v_6\}$. The cycle in $G-\{v_5v_6,v_4v_{10}\}$ has a common edge $v_2v_8$ with the even cycle $v_2v_8v_9v_{10}v_4v_3v_2$ and has a common edge $v_1v_8$ with the odd cycle $v_1v_8v_9v_{11}v_7v_1$. Hence, remove the edge $v_1v_8$ from $G-\{v_5v_6,v_4v_{10}\}$. Now all the cycles in $H=G-\{v_5v_6,v_4v_{10},v_1v_8\}$ are of even length, which is shown in Figure \ref{G-FG-MBSG}. Therefore, $H$ is the maximal bipartite subgraph in $G$. If $E'=\{v_5v_6,v_4v_{10},v_1v_8\}$, then by Theorem \ref{T-SNGB1}, $\varphi(G)=|E'|=3$.
\end{proof}

\begin{figure}[h!]
\begin{center}
\begin{subfigure}[b]{0.45\textwidth}
\includegraphics[scale=0.4]{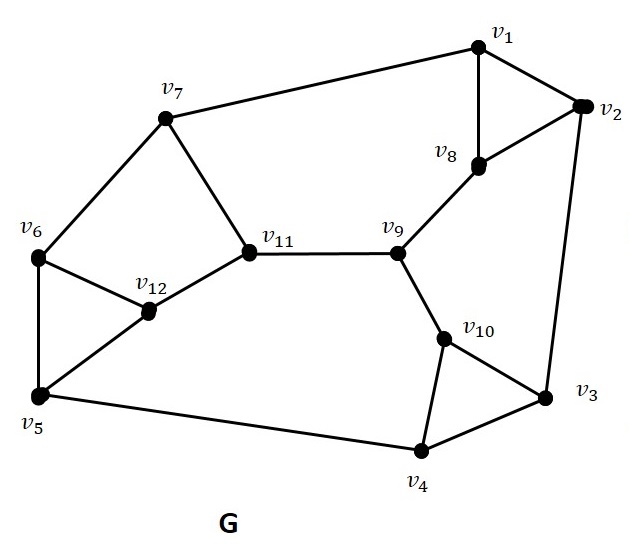}
\caption{\small \sl The Frucht graph.}\label{G-FGWL}
\end{subfigure}
\quad
\begin{subfigure}[b]{0.45\textwidth}
\includegraphics[scale=0.4]{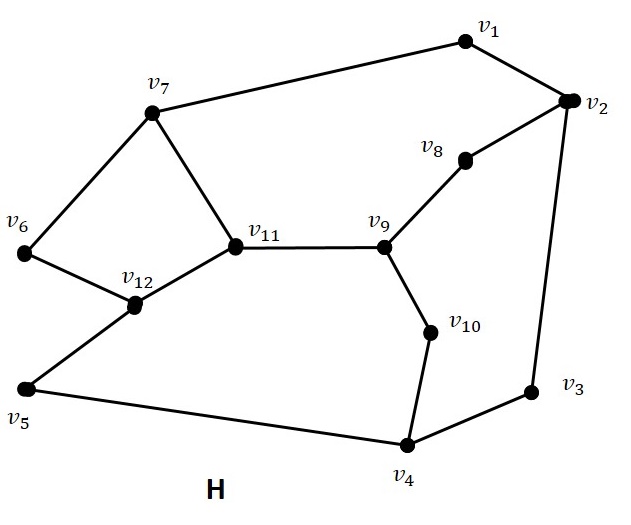}
\caption{\small \sl Maximal bipartite subgraph of the Frucht graph.}\label{G-FG-MBSG}
\end{subfigure}
\caption{}
\end{center}  
\end{figure}

\begin{remark}{\rm
For the Frucht graph $G$, $|E(G)|=18$. Hence, by Corollary \ref{C-T-SNGB1}, the number of edges in a maximal bipartite subgraph of $G$ is $b(G)= |E(G)|-\varphi(G)=18-3=15$.}
\end{remark}

The $Gr{\ddot{o}}tzsch ~ graph$ is a triangle-free graph with $11$ vertices, $20$ edges. The following theorem establishes the sparing number of {\rm $Gr{\ddot{o}}tzsch$} graph.

\begin{proposition}
The sparing number of the {\rm $Gr{\ddot{o}}tzsch$} graph is $5$.
\end{proposition}
\begin{proof}
Name the vertices of {\rm $Gr{\ddot{o}}tzsch$} graph $G$ as shown in Figure \ref{G-GGWL}. Remove the edge $v_3v_9$ common to the odd cycles $v_1v_2v_3v_9v_{10}v_1$ and $v_3v_4v_5v_6v_7v_8v_9v_3$, the edge $v_1v_5$ common to $v_1v_2v_3v_4v_5v_1$ and $v_5v_6v_7v_8v_9v_{10}v_1v_5$, the edge $v_3v_7$ common to $v_3v_4v_5v_6v_7v_3$ and $v_1v_2v_3v_4v_5v_6v_7v_8v_9v_{10}v_1$, the edge $v_5v_9$ common to $v_5v_6v_7v_8v_9v_5$ and the edge $v_9v_{10}v_1v_2v_3v_4\\v_5v_9$ and $v_1v_7$ common to $v_7v_8v_9v_{10}v_1v_7$ and $v_1v_2v_3v_4v_5v_6v_7v_1$.

The resultant graph $H=G-\{v_3v_9,v_1v_5,v_3v_7,v_5v_9,v_1v_7\}$ contains no odd cycles, as shown in \ref{G-GGMBSG}, and hence $H$ is a maximal bipartite subgraph of $G$. Let $E'=\{v_3v_9,v_1v_5,v_3\\v_7, v_5v_9,v_1v_7\}$. Then, the sparing number $\varphi(G)=|E'|=5$.
\end{proof}

\begin{figure}[h!]
\begin{center}
\begin{subfigure}[b]{0.45\textwidth}
\includegraphics[scale=0.5]{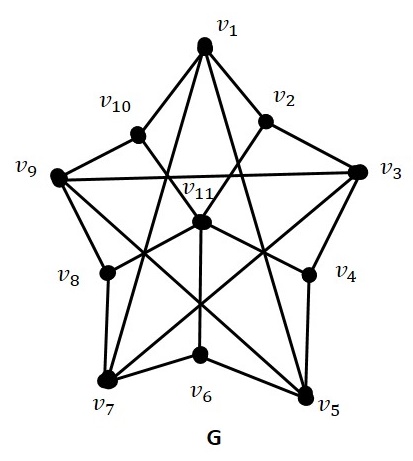}
\caption{\small \sl The {\rm $Gr{\ddot{o}}tzsch$} graph.}\label{G-GGWL}
\end{subfigure}
\quad
\begin{subfigure}[b]{0.45\textwidth}
\includegraphics[scale=0.5]{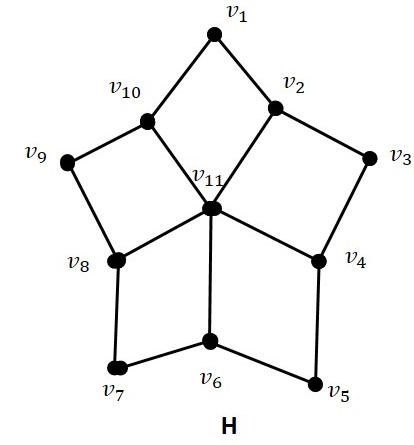}
\caption{\small \sl Maximal bipartite subgraph of\\ {\rm $Gr{\ddot{o}}tzsch$} graph.}\label{G-GGMBSG}
\end{subfigure}
\caption{}
\end{center}  
\end{figure}

\begin{remark}{\rm
For the {\rm $Gr{\ddot{o}}tzsch$} graph $G$, $|E(G)|=20$. Hence, by Corollary \ref{C-T-SNGB1}, the number of edges in a maximal bipartite subgraph of $G$ is $b(G)= |E(G)|-\varphi(G)=20-5=15$.}
\end{remark}

Another well-known non-bipartite graph is the ${Du{\ddot{r}}er graph}$, which is a $3$-regular graph with $12$ vertices, $18$ edges and every vertex of which can be distinguished topologically from every other vertex. Hence we have,

\begin{proposition}
The sparing number of {\em ${D{\ddot{u}}rer graph}$} is $4$.
\end{proposition}
\begin{proof}
Name the vertices of the {\em ${D{\ddot{u}}rer graph}$} $G$ as shown in Figure \ref{G-DGWL}. In $G$, the cycle $C_1:v_2v_4v_6$ and $C_2:v_3v_5v_1$ are of odd length. Hence, remove the edge $v_2v_6$ from $C_1$ and remove the edge $v_3v_5$ from $C_2$. The cycles $C_3:u_1v_1v_5u_5u_6u_1$ and $C_4:u_1v_1v_3u_3u_2u_1$ in $G-\{v_2v_6,v_3v_5\}$ are odd cycles and have a common edge $u_1v_1$ and the cycles $C_5:u_2v_2v_4u_4u_3u_2$ and $C_6:u_6v_6v_4u_4u_5u_6$ in $G-\{v_2v_6,v_3v_5\}$ are odd cycles and have a common edge $u_4v_4$. Hence, remove the edges $u_1v_1$ and $u_4v_4$ from $G-\{v_2v_6,v_3v_5\}$. Now all the cycles in $H=G-\{v_2v_6,v_3v_5,u_1v_1,u_4v_4\}$ are of even length, which is shown in \ref{G-DG-MBSG}. Therefore, $H$ is the maximal bipartite subgraph in $G$. If $E'=\{v_2v_6,v_3v_5,u_1v_1,u_4v_4\}$, then by Theorem \ref{T-SNGB1}, $\varphi(G)=|E'|=4$.
\end{proof}

\begin{figure}[h!]
\begin{center}
\begin{subfigure}[b]{0.4\textwidth}
\includegraphics[scale=0.4]{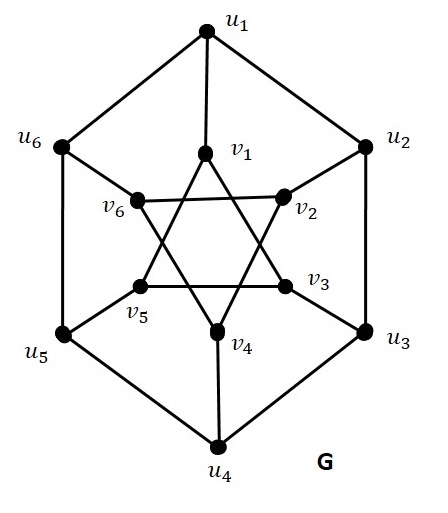}
\caption{\small \sl The {\em ${D{\ddot{u}}rer graph}$}.}\label{G-DGWL}
\end{subfigure}
\quad
\begin{subfigure}[b]{0.45\textwidth}
\includegraphics[scale=0.40]{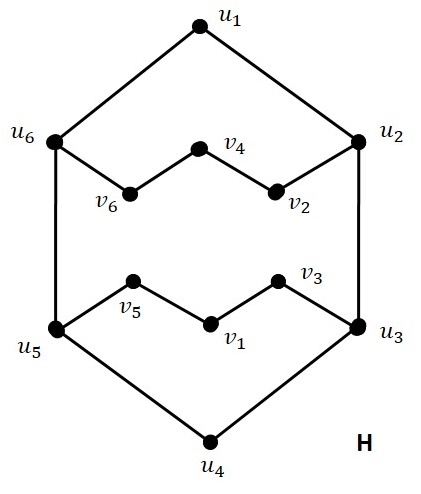}
\caption{\small \sl Maximal bipartite subgraph of the {\rm ${D\ddot{u}}rer$} graph.}\label{G-DG-MBSG}
\end{subfigure}
\caption{}
\end{center}  
\end{figure}

\begin{remark}{\rm
For the {\rm ${D\ddot{u}}rer$} graph $G$, $|E(G)|=18$. Hence, by Corollary \ref{C-T-SNGB1}, the number of edges in a maximal bipartite subgraph of $G$ is $b(G)= |E(G)|-\varphi(G)=18-4=14$.}
\end{remark}

A {\em Dodecahedron} is another popular non-bipartite graph. It is a $3$-regular graph with $20$ vertices and $30$ edges and it contains $12$ pentagons (the cycle $C_5$). Hence, we have

\begin{proposition}
The sparing number of the dodecahedron is $6$.
\end{proposition}
\begin{proof}
Name the vertices of dodecahedron $G$ as shown in Figure \ref{G-DGWL}. From $G$, remove the edge $u_7u_{17}$ common to the odd cycles $u_7u_{17}u_{16}u_{15}u_6u_7$ and $u_7u{17}u_{18}u_9u_8u_7$, the edge $u_{13}u_{20}$ common to $u_{13}u_{20}u_{16}u_{15}u_{14}u_{13}$ and $u_{13}u_{20}u_{19}u_{11}u_{12}u_{13}$, the edge $u_{18}u_{19}$ common to $u_{18}u_{19}u_{20}u_{16}u_{17}u_{18}$ and $u_{18}u_{19}u_{11}u_{10}u_9u_{18}$, the edge $u_{5}u_{6}$ common to $u_5u_6u_7u_8u_1u_5$ and $u_5u_6u{15}u_{14}u_4u_5$, the edge $u_2u_{10}$ common to $u_2u_{10}u_9u_8u_1u_2$ and $u_2u_{10}u{11}u_{12}u_3u_2$ and the edge $u_3u_4$ common to  $u_3u_4u_5u_1u_2u_3$ and $u_3u_4u{14}u_{13}u_{12}u_3$.

The resultant graph $H=G-\{u_7u_{17},u_{13}u_{20},u_{18}u_{19},u_{5}u_{6},u_2u_{10},u_3u_4\}$ contains no odd cycles, as shown in \ref{G-DDG-MBSG}, and hence $H$ is a maximal bipartite subgraph of $G$. Let $E'=\{u_7u_{17},u_{13}u_{20},u_{18}u_{19},u_{5}u_{6},u_2u_{10},u_3u_4\}$. Then, the sparing number $\varphi(G)=|E'|=6$.
\end{proof}

\begin{remark}{\rm
For a dodecahedron $G$, $|E(G)|=30$. Hence, by Corollary \ref{C-T-SNGB1}, the number of edges in a maximal bipartite subgraph of $G$ is $b(G)= |E(G)|-\varphi(G)=30-6=24$.}
\end{remark}

\begin{figure}[h!]
\begin{center}
\begin{subfigure}[b]{0.4\textwidth}
\includegraphics[scale=0.4]{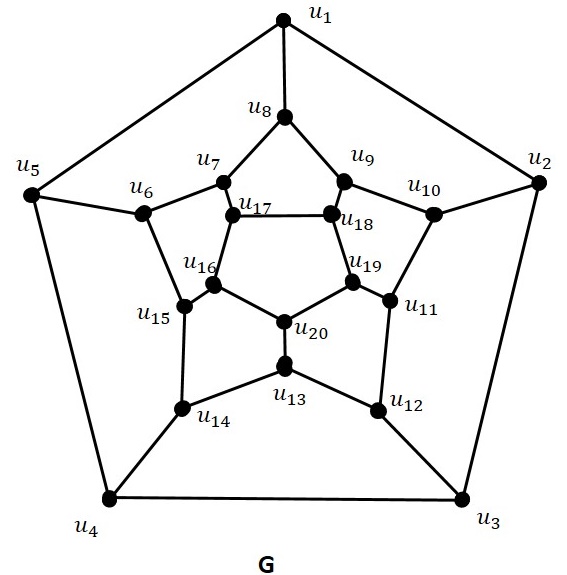}
\caption{\small \sl A dodecahedron.}\label{G-DDGWL}
\end{subfigure}
\quad
\begin{subfigure}[b]{0.4\textwidth}
\includegraphics[scale=0.4]{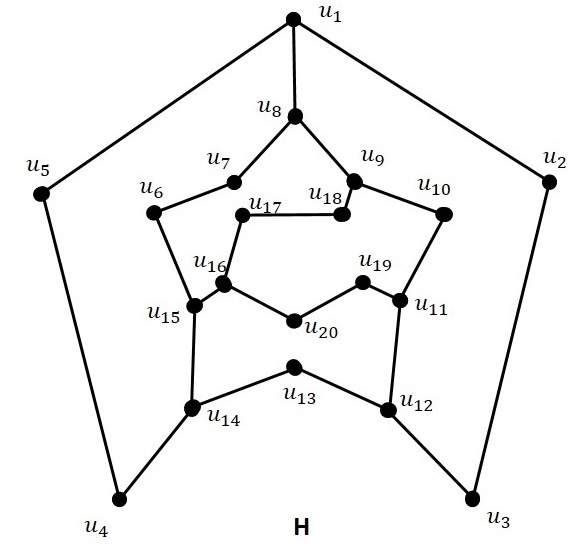}
\caption{\small \sl Maximal bipartite subgraph of a dodecahedron.}\label{G-DDG-MBSG}
\end{subfigure}
\caption{}
\end{center}  
\end{figure}

\section{Conclusion}
In this paper, we have established some results on the sparing number of certain graphs and established some relations between the sparing number and some other parameters of certain graph classes. The admissibility of weak IASI by various graph classes, graph operations and graph products and finding the corresponding sparing numbers are open.

More properties and characteristics of different types of IASIs, both uniform and non-uniform, are yet to be investigated. The problems of establishing the necessary and sufficient conditions for various graphs and graph classes to have certain IASIs are also open.


\begin{thebibliography}{25}
\bibitem {BM1} J A Bondy and U S R Murty, (2008). {\bf Graph Theory}, Springer.
\bibitem {CZ} G Chartrand and P Zhang, (2005). {\bf Introduction to Graph Theory}, McGraw-Hill Inc.
\bibitem {ND} N Deo, (1974). {\em Graph Theory with Applications to Engineering and Computer Science}, PHI Learning.
\bibitem {JAG1} J A Gallian, (2011). {\em A Dynamic Survey of Graph Labelling}, The Electronic Journal of Combinatorics (DS 16).
\bibitem {GA} K A Germina and T M K Anandavally, (2012). {\em Integer Additive Set-Indexers of a Graph:Sum Square Graphs}, Journal of Combinatorics, Information and System Sciences, {\bf 37}(2-4), 345-358.
\bibitem {GS1} K A Germina and N K Sudev, {\em On Weakly Uniform Integer Additive Set-Indexers of Graphs}, Int. Math Forum {\bf 8}(37), 1827-34.
\bibitem {GS2} K A Germina and N K Sudev, (2013). {\em Some New Results on Strong Integer Additive Set-Indexers}, communicated.
\bibitem {FH}  F Harary, (1994). {\bf Graph Theory}, Addison-Wesley Publishing Company Inc.
\bibitem {GS0} N K Sudev and K A Germina, (2014). {\em On Integer Additive Set-Indexers of Graphs}, To appear in Int. J. Math. Sci. \& Engg. Appl., {\bf 8}(2), .
\bibitem {GS3} N K Sudev and K A Germina, (2014).{\em A Characterisation of Weak Integer Additive Set-Indexers of Graphs}, ISPACS J. Fuzzy Set Valued Analysis, {\bf 2014}(2014)(), Article Id:jfsva-00189, 7 Pages.
\bibitem {GS4} N K Sudev and K A Germina, (2014). {\em Weak Integer Additive Set-Indexers of Certain Graph Operations}, To appear in Global J. Math Sci.: Theory and Practical. 
\bibitem {GS5} N K Sudev and K A Germina, {\em On the Sparing Number of Certain Graph Structures}, Communicated.
\bibitem {GS6} N K Sudev and K A Germina, {\em Weak Integer Additive Set-Indexers of Graph Products}, Communicated.
\bibitem {DBW} D B West, (2001). {\bf Introduction to Graph Theory}, Pearson Education Inc.
\end{thebibliography}
\end{document}